\documentclass{amsproc}
\usepackage{amsmath,amssymb,amsthm}
\usepackage{graphics}

\theoremstyle{definition}

\theoremstyle{remark}

\numberwithin{equation}{section}
\usepackage{enumerate}
\usepackage{graphicx}

\newtheorem*{satz}{Theorem}
\newtheorem*{lemm}{Lemma}

\begin{document}

\title{Lexicographic Generation of Projective Spaces}
\author{Christoph Hering}
\address{Institute of Mathematics, University of T\"ubingen, 
Auf der Morgenstelle 10, D72076 T\"ubingen, Germany}
\email{hering@uni-tuebingen.de}

\author{Hans-J\"org Schaeffer}
\address{Rossbergstr. 47, D72072 T\"ubingen, Germany}
\email{hjs@hjschaeffer.de}

\begin{abstract} Lexicographic or first choice constructions of geometric objects sometimes lead to amazingly 
good results. Usually it is difficult to determine the precise identity of these geometries. Here we find infinitely many cases where the identification actually can be accomplished. 
\end{abstract}



\maketitle

Let $k,r \in \mathbb{N}$. There exists exactly one $\mathbb{N} \times \mathbb{N}$-matrix $A=(a_{ij})_{i,j\in\mathbb{N}}$ over $\{0,1\}$ such that $a_{ij} = 1$ if and only if none of the following conditions holds,
\begin{itemize}
  \item there exist $i' < i$ and $j' < j$ such that $a_{i,j'} = 
a_{i',j} = a_{i',j'} = 1;$
  \item $\sum_{j'<j} a_{i,j'} \geq k;$
  \item $\sum_{i'<i} a_{i',j} \geq r.$
\end{itemize}
This matrix is called the \emph{naive matrix} of Type $(k,r)$. (In this paper $\mathbb{N} = \{1,2,3,\cdots \}$).

\begin{satz}
Assume $k=3$ and $r=2^n-1, n \geq 1$ and denote 
\[
d=(2^{n+1}-1)(2^{n}-1)/3.
\]
Then the first $d$ rows of $A$ are of the form 
\[
\{a,b,a \oplus b\}, \mbox{ where } ~ 0<a<b<a\oplus b <2^{n+1}
\]
and $\oplus$ is the Nim addition.
\end{satz}
\begin{proof}
Let $L_i$ be the $i$-th row of our matrix A. Clearly $L_1=\{1,2,3\}=\{1,2,1\oplus 2 \}$. (We identify the rows of $A$ with their support, so $L_i = ~\{j ~|~ a_{i,j}=1\}.)$ Let $1\leq m <d $ and assume that $L_1,\cdots,L_m$ are of the form given in the theorem. 
Denote $\mathcal{L}=\{L_1,\cdots,L_m\}$ and $L_{m+1}=\{a,b,c\}$, where $a<b<c$.

Let $\mathcal{P}_0 = \{1,\cdots,2^{n+1}-1\}.$ We call a point $i\in \mathbb{N}$ \emph{complete (with respect to} $\mathcal{L}$) if $\sum_{j \leq m} a_{j,i} = r$. Because $\mid \mathcal{P}_0 \mid\cdot ~ r = d \cdot k > m k $, the points in $\mathcal{P}_0$ cannot all be complete. Therefore $a \in \mathcal{P}_0 $.

Two points $i,j \in \mathbb{N}$ are called \emph{connectable (with respect to} $\mathcal{L}$) if there exists a row in $\mathcal{L}$ containing $i $ and $j$. Note that $\mid \mathcal{P}_0 \mid ~= r(k-1)+1.$ Therefore

(1) If a point $x$ in $\mathcal{P}_0$ is complete, then it is connectable to any other point in $\mathcal{P}_0.$ 

We mention some further properties of our matrix: 

If $1\leq x < c$ and $x\neq a,b$, then $x$ is complete or connectable to $a$ or $b$. Together with (1) we have

(2) If $1\leq x <c$ and $x\leq 2^{n+1}-1$, then $x$ is connectable to $a$ or $b$.

In the same way 

(3) If $1 \leq x < b$ and $x\leq 2^{n+1}-1$, then $x$ is connectable to $a$.

The rows (lines) in $\mathcal{L}$ through $a$ cover at most $(r-1) (k-1) +1 =2^{n+1}-3$ points. A point not covered by these rows cannot be complete by $(1)$. Therefore $b \in \mathcal{P}_0$. Furthermore $a,b<2^{n+1}$ implies $a \oplus b<2^{n+1}$, and hence $a\oplus b\in \mathcal{P}_0$ also. 

Suppose $a\oplus b<c.$ If $a\oplus b$ is connectable to $a$, then $\mathcal{L}$ contains the row $\{a,a\oplus b,b\}$, as $a\oplus (a\oplus b)=b$, a contradiction. Hence $a\oplus b$ is not connectable to $a$ and therefore connectable to $b$ by $(2)$.
 But then $\mathcal{L}$ contains the row $\{b,a\oplus b,a\}$, again a contradiction. Therefore $c \leq a\oplus b<2^{n+1}$, and $c \in \mathcal{P}_0$.

Suppose now that $c<a\oplus b$. We use the greediness of the Nim addition. We have $c \oplus a < b $ or $c \oplus b < a $ by the Greediness Lemma 
%\ref{lem:1} 
below. In the first case $c \oplus a $ is connectable to $a$ by $(3)$, and $\mathcal{L} $ contains the row $\{ a , c \oplus a , (c\oplus a ) \oplus a \} = \{ a , c \oplus a , c \}, $ a contradiction. In the second case $c \oplus b $ is complete and hence connectable to $b$ by the row $\{ c \oplus b , b , ( c \oplus b ) \oplus b \} = \{ c \oplus b , b , c \}, $ which again is a contradiction. 

Therefore there only remains the case $c = a \oplus b $. 
\end{proof}

\begin{lemm}\label{lem:1} Greediness Lemma. Let $ a , b , c \in \mathbb {N} \cup \{0\}.$ If $ c < a \oplus b,$ then $ a \oplus c < b $ or $ b \oplus c < a. $

\end{lemm}

\begin{proof} Let $a = (a_0 , a_1 , \cdots ),~ b = (b_0 , b_1 \cdots ) $ and $~c = (c_0 , c_1 , \cdots )$ be the binary expansions (so, e.g., $a = a_0 + a_1 2^1 + a_2 2^2 + \cdots, $ and $a_i \in \{0,1\}$). Let $j$ be the largest number such that $ c_j \neq a_j \oplus b_j$. Then $c < a \oplus b $ implies $c_j = 0 $ and $a_j \oplus b_j = 1 $ so that $a_j = 1 , ~b_j = 0 $ or $b_j = 1 , ~a_j = 0. $ Assume at first $a_j = 1 $ and $b_j = 0. $ For $k > j $ we have $c_k = a_k \oplus b_k $ and hence $a_k = b_k \oplus c_k. $ On the other hand the $j- $th coefficient of $c \oplus b $ is $c_j \oplus b_j = 0, $ while $a_j = 1. $ So $c \oplus b < a. $ 

In the same way, the second possibility $b_j = 1 $ and $a_j = 0 $ implies $c \oplus a < b$.

Here, again, $ \oplus $ stands for the Nim addition.
\end{proof}

$\mathcal{P}_0 \cup \{0\} $ is a subgroup of order $2^{n+1} $ of the elementary abelian $2-$ group $(\mathbb{N} , \oplus).$ The row $ \{a,b,a \oplus b\} $ above consists of the non-zero elements of a subgroup order $4$ of $\mathcal{P}_0 \cup \{0\} . $ As $d$ is the number of subgroups of order $4$ of $\mathcal{P}_0 \cup \{0\}, $ the rows $L_1 , \cdots , L_d$ correspond to the set of all subgroups of order $4$ of this group. Therefore the upper left hand corner of our matrix contains the incidence matrix of a projective space $PG ~ (n,2),$ the columns corresponding to points and the rows to lines. 

Clearly $a_{i,j} = 0 $ for $1\leq i \leq d $ and $j > 2^{n+1} - 1 : = s $ and also for $i > d $ and $1\leq j \leq s ~ (as ~ d \cdot k = d \cdot 3 = s \cdot r). $ Therefore $a _{i + d, j + s} = a _{i j} $ for $i , j \in \mathbb{N}.$

Naive matrices of the form presented here arise from lexicographic or first choice constructions of $(k , r)$-configurations.The Greediness Lemma above is related to results of Conway in \cite {Co}. For a more detailed analysis and further literature see \cite {KHE} and \cite {KHE1}. 

Using some more algebra we can generalize the Theorem to $k=q+1$ and $r=(q^n-1)/(q-1)$, where $q=2^{2^{a}}, a\geq 0$ is a Fermat 2-power. We then obtain incidence matrices of the point-line designs of projective spaces 
%of dimension $n$ 
over $GF(q)$.

\end{document}